\documentclass[11pt]{article}

\usepackage{epsfig}
\usepackage{amssymb}
\usepackage[]{times}
\newcommand{\be}{\begin{equation}}
\newcommand{\ee}{\end{equation}}
\newcommand{\bea}{\begin{eqnarray}}
\newcommand{\eea}{\end{eqnarray}}

\newtheorem{corollary}{Corollary}
\newtheorem{lemma}{Lemma}
\newtheorem{conjecture}{Conjecture}
\newenvironment{proof}[1][Proof]{\begin{trivlist}
\item[\hskip\labelsep {\bfseries #1}]}{\end{trivlist}}
\def\1#1{^{(#1)}}
\def\la{\langle}
\def\ra{\rangle}
\begin{document}
\title{Symmetric polynomials associated with numerical semigroups}
\author{Leonid G. Fel\\ \\
Department of Civil Engineering, Technion, Haifa 32000, Israel\\
{\em e-mail: lfel@technion.ac.il}} 
\date{}
\maketitle
\def\be{\begin{equation}}
\def\ee{\end{equation}}
\def\bea{\begin{eqnarray}}
\def\eea{\end{eqnarray}}
\def\p{\prime}
\vspace{-1cm}
%%%%%%%%%%%%%%%%%%%%%%%%%%%%%%%%%%%
\begin{abstract}
We study a new kind of symmetric polynomials $P_n(x_1,\ldots,x_m)$ of degree $n$
in $m$ real variables, which have arisen in the theory of numerical semigroups. 
We establish their basic properties and find their representation through the 
power sums $E_k=\sum_{j=1}^mx_j^k$. We observe a visual similarity between 
normalized polynomials $P_n(x_1,\ldots,x_m)/\chi_m$, where $\chi_m\!=\!\prod_{
j=1}^mx_j$, and a polynomial part of a partition function $W\left(s,\{d_1,
\ldots,d_m\}\right)$, which gives a number of partitions of $s\ge 0$ into $m$ 
positive integers $d_j$, and put forward a conjecture about their relationship.
\end{abstract}
%%%%%%%%%%%%%%%%%%%%%%%%%%%%%%%%%%%%%%%%%%%%%%%%%
\section{Symmetric polynomials $P_n({\bf x}^m)$ and their factorization}
\label{l1}
%%%%%%%%%%%%%%%%%%%%%%%%%%%%%%%%%%%%%%%%%%%%%%%%%
In 2017, studying the polynomial identities of arbitrary degree for syzygies 
degrees of numerical semigroups $\la d_1,\ldots,d_m\ra$, we have introduced a 
new kind of symmetric polynomials $P_n(x_1,\ldots,x_m)$ of degree $n$ in $m$ 
real variables $x_j$ (see \cite{fl17}, section 5.1),
\bea
P_n({\bf x}^m)=\sum_{j=1}^mx_j^n-\sum_{j>r=1}^m\left(x_j+x_r\right)^n+\sum_{j
>r>i=1}^m\left(x_j+x_r+x_i\right)^n-\ldots-(-1)^m\left(\sum_{j=1}^mx_j\right)^n,
\label{j1}
\eea
where ${\bf x}^m$ denotes a tuple $\{x_1,\ldots,x_m\}$ and $P_n({\bf x}^m)$ is 
invariant under the action of the symmetric group $S_m$ on a set of variables 
$\{x_1,\ldots,x_m\}$ by their permutations. Such polynomials arise in the 
rational representation of the Hilbert series for the complete intersection 
semigroup ring associated with a symmetric semigroup $\la d_1,\ldots,d_m\ra$. 
According to \cite{fl17}, the polynomials in (\ref{j1}) satisfy 
\bea
P_n({\bf x}^m)=0,\quad 1\le n\le m-1\qquad\mbox{and}\qquad P_m({\bf x}^m)=
(-1)^{m+1}m!\prod_{j=1}^mx_j.\label{j2}
\eea

In this paper, we study a factorization of $P_n({\bf x}^m)$ for $n>m$ and make 
use of this property to find a representation of $P_n({\bf x}^m)$ through the 
power sums $E_k=\sum_{j=1}^mx_j^k$, i.e., $P_n({\bf x}^m)=P_n(E_1,\ldots,E_n)$.
%%%%%%%%%%%%%%%%%%%%%%%%%%%%%%%%%%%%%%%%%%%%%%%%%
\begin{lemma}\label{le1}
$P_n({\bf x}^m)$ vanishes if at least one of the variables $x_j$ vanishes.
\end{lemma}
%%%%%%%%%%%%%%%%%%%%%%%%%%%%%%%%%%%%%%%%%%%%%%%%%%%%%%%%
\begin{proof}
Since $P_n({\bf x}^m)$ is invariant under all permutations of variables $\{x_1,
\ldots,x_m\}$, we have to prove
\bea
P_n(0,x_2,\ldots,x_m)=0.\label{j3}
\eea
Denote $P_n(0,x_2,\ldots,x_m)=P_n(0,{\bf x}^{m-1})$ and substitute $x_1=0$ into 
(\ref{j1}),
\bea
P_n(0,{\bf x}^{m-1})\!\!\!&\!\!=\!\!&\!\!\!\sum_{j=2}^mx_j^n\!-\!\left[\sum_{j
=2}^mx_j^n+\!\sum_{j>r=2}^m\!\left(x_j+x_r\right)^n\right]\!+\!\left[\sum_{j>r
=2}^m\!\left(x_j+x_r\right)^n+\!\!\sum_{j>r>i=2}^m\!\!\!\left(x_j+x_r+x_i
\right)^n\right]\nonumber\\
&-&\left[\sum_{j>r>i=2}^m\!\!\!\left(x_j+x_r+x_i\right)^n+\sum_{t>j>r>i=2}^m
\!\!\!\left(x_t+x_j+x_r+x_i\right)^n\right]+\ldots\nonumber\\
&+&(-1)^m\left[\sum_{j=2}^m\left(\sum_{r=2}^mx_j+x_r\right)^n+\left(\sum_{j=2}^m
x_j\right)^n\right]-(-1)^m\left(\sum_{j=2}^mx_j\right)^n.\nonumber
\eea
Recasting the terms in the last sum in $m$ pairs, we obtain,
\bea
P_n(0,{\bf x}^{m-1})\!\!\!&\!\!=\!\!&\!\!\!\left[\sum_{j=2}^mx_j^n\!-\sum_{j=2}
^mx_j^n\right]\!-\!\left[\sum_{j>r=2}^m\!\left(x_j\!+\!x_r\right)^n-\sum_{j>r=2}
^m\!\left(x_j\!+\!x_r\right)^n\right]\!+\!\left[\sum_{j>r>i=2}^m\!\!\!\left(x_j
\!+\!x_r\!+\!x_i\right)^n\right.\nonumber\\
&-&\left.\sum_{j>r>i=2}^m\left(x_j+x_r+x_i\right)^n\right]-\ldots+(-1)^m\left[
\left(\sum_{j=2}^mx_j\right)^n-\left(\sum_{j=2}^mx_j\right)^n\right]=0,\nonumber
\eea
and Lemma is proven.$\;\;\;\;\;\;\Box$
\end{proof}
%%%%%%%%%%%%%%%%%%%%%%%%%%%%%%%%%%%%%%%%%%%%%%%%%
\begin{corollary}\label{cor1}
$P_n({\bf x}^m)$ is factorizable by the product $\chi_m=\prod_{j=1}^mx_j$.
\end{corollary}
%%%%%%%%%%%%%%%%%%%%%%%%%%%%%%%%%%%%%%%%%%%%%%%%%%%%%%%%
\begin{proof}
Since $P_n({\bf x}^m)$ is invariant under all permutations of variables $\{x_1,
\ldots,x_m\}$, then Lemma \ref{le1} is true if we replace $x_1=0$ by any other 
variable $x_j$, i.e.,
\bea
P_n(0,x_2,\ldots,x_m)=P_n(x_1,0,\ldots,x_m)=\ldots=P_n(x_1,x_2,\ldots,0)=0.
\label{j4}
\eea
Thus, equation $P_n(x_1,\ldots,x_m)=0$ has, at least, $m$ independent roots $x_1
=x_2=\ldots=x_m=0$. Then, by the polynomial factor theorem, $P_n({\bf x}^m)$ is 
factorizable by the product $\chi_m$.$\;\;\;\;\;\;\Box$
\end{proof}
%%%%%%%%%%%%%%%%%%%%%%%%%%%%%%%%%%%%%%%%%%%%%%%%%
In full agreement with (\ref{j2}), by Corollary \ref{cor1} it follows that $P_n(
{\bf x}^m)=0$ if $n<m$ and $P_m({\bf x}^m)=const$.
%%%%%%%%%%%%%%%%%%%%%%%%%%%%%%%%%%%%%%%%%%%%%%%%%
\begin{lemma}\label{le2}
$P_n({\bf x}^m)$ is factorizable by a sum $E_1=\sum_{j=1}^mx_j$ if $n-m=1(\bmod
\;2)$.
\end{lemma}
%%%%%%%%%%%%%%%%%%%%%%%%%%%%%%%%%%%%%%%%%%%%%%%%%%%%%%%%
\begin{proof}
Rewrite $P_n({\bf x}^m)$ as follows,
\bea
P_n({\bf x}^m)&=&\sum_{j=1}^mx_j^n-\sum_{j_1>j_2=1}^m\!\left(\sum_{k=1}^2x_{j_k}
\right)^n+\sum_{j_1>j_2>j_3=1}^m\!\left(\sum_{k=1}^3x_{j_k}\right)^n-\ldots
\label{j5}\\
&-&(-1)^m\sum_{j_1>j_2=1}^m\left(E_1-\sum_{k=1}^2x_{j_k}\right)^n+(-1)^m
\sum_{j=1}^m\left(E_1-x_j\right)^n-(-1)^mE_1^n,\nonumber
\eea
and substitute there $E_1=0$,
\bea
P_n({\bf x}^m)&=&\sum_{j=1}^mx_j^n-\sum_{j_1>j_2=1}^m\!\left(\sum_{k=1}^2x_{j_k}
\right)^n+\sum_{j_1>j_2>j_3=1}^m\!\left(\sum_{k=1}^3x_{j_k}\right)^n-\ldots
\nonumber\\
&+&(-1)^{m+n}\sum_{j_1>j_2>j_3=1}^m\!\left(\sum_{k=1}^3x_{j_k}\right)^n-(-1)^
{m+n}\sum_{j_1>j_2=1}^m\!\left(\sum_{k=1}^2x_{j_k}\right)^n+(-1)^{m+n}
\sum_{j=1}^mx_j^n.\nonumber
\eea
Recast the terms in the last sum as follows,
\bea
P_n({\bf x}^m)&=&\left[1+(-1)^{m+n}\right]R_{1.n}({\bf x}^m)+\frac{(-1)^{\mu}}
{2}\left[ 1+(-1)^m\right]R_{2,n}({\bf x}^m),\label{j6}\\
R_{1.n}({\bf x}^m)&=&\sum_{j=1}^mx_j^n-\sum_{j_1>j_2=1}^m\!\left(\sum_{k=1}^2
x_{j_k}\right)^n+\ldots-(-1)^{\mu}\!\sum_{j_1>j_2>\ldots>j_{\mu}=1}^m\!
\left(\sum_{k=1}^{\mu}x_{j_k}\right)^n,\nonumber\\
R_{2,n}({\bf x}^m)&=&\sum_{j_1>j_2>\ldots>j_{\mu+1}=1}^m\!\left(\sum_{k=1}^{\mu+1}
x_{j_k}\right)^n,\qquad\mu=\left\lfloor\frac{m-1}{2}\right\rfloor,\label{j7}
\eea
where $\lfloor a\rfloor$ denotes the integer part of $a$. 

According to (\ref{j6}), if $m+n=1(\bmod\;2)$ and $m=1(\bmod\;2)$, then $P_n(
{\bf x}^m)=0$. Consider another case when $m+n=1(\bmod\;2)$ and $m=0(\bmod\;2)$.
Put $m=2q$ and $n=2l+1$ in (\ref{j6},\ref{j7}) and obtain
\bea
P_{2l+1}\left({\bf x}^{2q}\right)=(-1)^{q-1}\sum_{j_1>j_2>\ldots>j_q=1}^{2q}
\!\left(\sum_{k=1}^qx_{j_k}\right)^{2l+1}.\label{j8}
\eea
In (\ref{j8}), a summation in the external sum $\sum_{j_1>j_2>\ldots>j_q=1}^{
2q}$ runs over all $(2q)!/(q!)^2$ permutations of $2q$ variables $x_j$ in terms 
$\left(\sum_{k=1}^qx_{j_k}\right)^{2l+1}$. That is why every such term has in 
(\ref{j8}) its counterpart,
\bea
&&\left(x_{j_1}+x_{j_2}+\ldots+x_{j_q}\right)^{2l+1}\quad\longleftrightarrow
\quad\left(x_{i_1}+x_{i_2}+\ldots+x_{i_q}\right)^{2l+1},\nonumber\\
&&\left\{x_{j_1},\ldots,x_{j_q}\right\}\cap\left\{x_{i_1},\ldots,x_{i_q}\right\}
=\emptyset,\nonumber\\
&&\#\left\{x_{j_1},\ldots,x_{j_q}\right\}=\#\left\{x_{i_1},\ldots,x_{i_q}
\right\}=q,\nonumber\\
&&\sum_{k=1}^qx_{j_k}+\sum_{k=1}^qx_{i_k}=E_1.\label{j9}
\eea
Recomposing the external sum in (\ref{j8}) as a sum over pairs, described in 
(\ref{j9}),
\bea
%\left(x_{j_1}+x_{j_2}+\ldots+x_{j_q}\right)^{2l+1}+
%\left(x_{i_1}+x_{i_2}+\ldots+x_{i_q}\right)^{2l+1},\qquad
\left(\sum_{k=1}^qx_{j_k}\right)^{2l+1}+\left(\sum_{k=1}^qx_{i_k}\right)^{2l+1}
\nonumber
\eea
and making use of the last equality in (\ref{j9}), where $E_1=0$, we arrive at
$P_{2l+1}\left({\bf x}^{2q}\right)=0$.

Thus, the polynomial $P_n\left({\bf x}^m\right)$ is factorizable by $E_1$ if 
$n+m=1(\bmod\;2)$. That finishes the proof of Lemma since two modular 
equalities, $n+m=1(\bmod\;2)$ and $n-m=1(\bmod\;2)$, are identically equivalent.
$\;\;\;\;\;\;\Box$
\end{proof}
%%%%%%%%%%%%%%%%%%%%%%%%%%%%%%%%%%%%%%%%%%%%%%%%%
\begin{lemma}\label{le3}
If $x_i>0$ then $P_n\left({\bf x}^m\right)$ satisfies the following inequalities,
\bea
P_n\left({\bf x}^m\right)>0,\quad m=1(\bmod\;2),\qquad 
P_n\left({\bf x}^m\right)<0,\quad m=0(\bmod\;2).
\label{j10}
\eea
\end{lemma}
%%%%%%%%%%%%%%%%%%%%%%%%%%%%%%%%%%%%%%%%%%%%%%%%%%%%%%%%
\begin{proof}
Prove (\ref{j10}) by induction. First, start with three simple inequalities,
\bea
P_n\left({\bf x}^2\right)\!\!\!&=&\!\!\!x_1^n+x_2^n-\left(x_1+x_2\right)^n<0,
\qquad n\ge 2,\quad x_1,x_2>0,\label{j11}\\
P_n\left({\bf x}^3\right)\!\!\!&=&\!\!\!-\sum_{k=1}^{n-1}{n\choose k}x_3^{n-k}
x_1^k-\sum_{k=1}^{n-1}{n\choose k}x_3^{n-k}x_2^k+\sum_{k=1}^{n-1}{n\choose k}
x_3^{n-k}\left(x_1+x_2\right)^k\nonumber\\
\!\!\!&=&\!\!\!-\sum_{k=1}^{n-1}{n\choose k}x_3^{n-k}P_k\left({\bf x}^2\right)
>0,\qquad n\ge 3,\quad x_1,x_2,x_3>0,\nonumber\\
P_n\left({\bf x}^4\right)\!\!\!&=&\!\!\!-\sum_{k=1}^{n-1}{n\choose k}x_4^{n-k}
x_1^k-\!\sum_{k=1}^{n-1}{n\choose k}x_4^{n-k}x_2^k-\!\sum_{k=1}^{n-1}{n\choose 
k}x_4^{n-k}x_3^k+\!\sum_{k=1}^{n-1}{n\choose k}x_4^{n-k}\left(x_1+x_2\right)^k\!
+\nonumber\\
&&\;\sum_{k=1}^{n-1}{n\choose k}x_4^{n-k}\left(x_2+x_3\right)^k+
\sum_{k=1}^{n-1}{n\choose k}x_4^{n-k}\left(x_3+x_1\right)^k-
\sum_{k=1}^{n-1}{n\choose k}x_4^{n-k}\left(x_1+x_2+x_3\right)^k\nonumber\\
\!\!\!&=&\!\!\!-\sum_{k=1}^{n-1}{n\choose k}x_4^{n-k}P_k\left({\bf x}^3\right)
<0,\qquad n\ge 4,\quad x_1,x_2,x_3,x_4>0.\nonumber
\eea
Next, establish an identity for $P_n\left({\bf x}^m\right)$ relating the last 
one with symmetric polynomials $P_k\left({\bf x}^{m-1}\right)$ of a smaller 
tuple,
\bea
P_n\left({\bf x}^m\right)=-\sum_{k=1}^{n-1}{n\choose k}x_m^{n-k}P_k\left({\bf x}
^{m-1}\right),\quad {\bf x}^{m-1}=\{x_1,\ldots,x_{m-1}\},\quad x_i>0,\label{j12}
\eea
which follows by careful recasting the terms in (\ref{j1}) and further 
simplification. 

Thus, according to (\ref{j12}), if $P_k\left({\bf x}^{m-1}\right)>0$, $x_i>0$, 
irrespectively to $k$, then $P_k\left({\bf x}^m\right)<0$, $x_i>0$, and vice 
versa, if $P_k\left({\bf x}^{m-1}\right)<0$, $x_i>0$, then $P_k\left({\bf x}^m
\right)>0$, $x_i>0$. On the other hand, in (\ref{j11}) the first terms of the 
alternating sequence $P_n({\bf x}^m)$ with growing $m$ satisfy (\ref{j10}). 
Then, by induction, inequalities (\ref{j10}) hold for every $m$.$\;\;\;\;\;\;
\Box$
\end{proof}
%%%%%%%%%%%%%%%%%%%%%%%%%%%%%%%%%%%%%%%%%%%%%%
\section{Representation of the polynomial $P_n({\bf x}^m)$}\label{l2}
%%%%%%%%%%%%%%%%%%%%%%%%%%%%%%%%%%%%%%%%%%%%%%
To provide $P_n({\bf x}^m)$ with properties (\ref{j2}) and satisfy Corollary 
\ref{cor1}, we choose the following representation for the polynomial,
\bea
P_n({\bf x}^m)=\frac{(-1)^{m+1}n!}{(n-m)!}\;\chi_mT_{n-m}\left({\bf x}^m\right),
\label{j13}
\eea
where $T_r\left({\bf x}^m\right)$ is a symmetric polynomial of degree $r$ in $m$
variables $x_j$. Combining (\ref{j13}) and Lemma \ref{le3}, we obtain
\bea
T_r\left({\bf x}^m\right)>0,\qquad x_i>0.\label{j14}
\eea

A straightforward calculation (with help of Mathematica software) of the eight 
first polynomials $T_r({\bf x}^m)$ results in the following expressions, 
satisfied Lemma \ref{le2},
\bea
T_0({\bf x}^m)&=&1,\label{j15}\\
T_1({\bf x}^m)&=&\frac1{2}E_1,\nonumber\\
T_2({\bf x}^m)&=&\frac1{3}\frac{3E_1^2+E_2}{4},\nonumber\\
T_3({\bf x}^m)&=&\frac1{4}\frac{E_1^2+E_2}{2}\;E_1,\nonumber\\
T_4({\bf x}^m)&=&\frac1{5}\frac{15E_1^4+30E_1^2E_2+5E_2^2-2E_4}{48},\nonumber\\
T_5({\bf x}^m)&=&\frac1{6}\frac{3E_1^4+10E_1^2E_2+5E_2^2-2E_4}{16}\;E_1,
\nonumber\\
T_6({\bf x}^m)&=&\frac1{7}\frac{63E_1^6+315E_1^4E_2+315E_1^2E_2^2-126E_1^2E_4+
35E_2^3-42E_2E_4+16E_6}{576},\nonumber\\
T_7({\bf x}^m)&=&\frac1{8}\frac{9E_1^6+63E_1^4E_2+105E_1^2E_2^2-42E_1^2E_4+
35E_2^3-42E_2E_4+16E_6}{144}\;E_1.\nonumber
\eea
Formulas (\ref{j15}) for $T_r({\bf x}^m)$ are valid irrespective to the ratio $r
/m$, or, in other words, to the fact how many power sums $E_k$ are algebraically
independent. In fact, if $r>m$ then expressions may be compactified by 
supplementary relations $E_k\!=\!E_k(E_1,\ldots,E_m)$, $k>m$. In section 
\ref{l4} we give such relations for small $m=1,2,3$.

Unlike to elementary symmetric polynomials $\sum_{i_1<i_2<\ldots<i_r}^mx_{i_1}
x_{i_2}\ldots x_{i_r}$ and power sums $E_r({\bf x}^m)$, the symmetric 
polynomials $T_r({\bf x}^m)$, $0\le r\le 7$, are algebraically dependent. 
Indeed, by (\ref{j15}) we get
\bea
\frac{T_3({\bf x}^m)}{T_1^3({\bf x}^m)}&=&3\frac{T_2({\bf x}^m)}{T_1^2({\bf x}
^m)}-2,\label{j16}\\
\frac{T_5({\bf x}^m)}{T_1^5({\bf x}^m)}&=&5\frac{T_4({\bf x}^m)}{T_1^4({\bf x}
^m)}-20\frac{T_2({\bf x}^m)}{T_1^2({\bf x}^m)}+16,\nonumber\\
\frac{T_7({\bf x}^m)}{T_1^7({\bf x}^m)}&=&7\frac{T_6({\bf x}^m)}{T_1^6({\bf x}
^m)}-70\frac{T_4({\bf x}^m)}{T_1^4({\bf x}^m)}+336\frac{T_2({\bf x}^m)}{T_1^2(
{\bf x}^m)}-272.\nonumber
\eea

It is unlikely to arrive at a general formula for $T_r({\bf x}^m)$ with 
arbitrary $r$ by observation of the fractions in (\ref{j15}). However, one can 
recognize a visual similarity between (\ref{j15}) and the other known 
expressions of special polynomials arisen in the theory of partition 
\cite{rf06}.

Recall formulas for a polynomial part $W_1\!\left(s,{\bf d}^m\right)$ of a 
restricted partition function $W\left(s,{\bf d}^m\right)$, where ${\bf d}^m=
\{d_1,\ldots,d_m\}$, which gives a number of partitions of $s\ge 0$ into $m$ 
positive integers $(d_1,\ldots,d_m)$, each not exceeding $s$, and vanishes, if 
such partition does not exist. Following formulas (3.16), (7.1) in \cite{fl17}, 
we obtain
\bea
W_1\!\left(s,{\bf d}^m\right)\!=\!\frac{1}{(m-1)!\;\pi_m}\sum_{r=0}^{m-1}{m-1
\choose r}f_r({\bf d}^m)s^{m-1-r},\qquad f_r({\bf d}^m)\!=\!\left(\sigma_1+
\sum_{i=1}^m{\cal B}\:d_i\right)^r,\label{j17}
\eea
where $\pi_m\!=\!\prod_{j=1}^md_j$ and $\sigma_1\!=\!\sum_{j=1}^md_j$. In 
(\ref{j17}) formula for $f_r({\bf d}^m)$ presumes a symbolic exponentiation 
\cite{rr78}: after binomial expansion the powers $({\cal B}\:d_i)^r$ are 
converted into the powers of $d_i$ multiplied by Bernoulli's numbers ${\cal B}_
r$, i.e., $d_i^r{\cal B}_r$. A  straightforward calculation of eight first 
polynomials $f_r({\bf d}^m)\!=\!f_r(\sigma_1,\ldots,\sigma_r)$ in terms of power
sums $\sigma_k\!=\!\sum_{j=1}^md_j^k$ were performed in \cite{fl17}, formulas 
(7.2),
\bea
f_0({\bf d}^m)&=&1,\label{j18}\\
f_1({\bf d}^m)&=&\frac1{2}\sigma_1,\nonumber\\
f_2({\bf d}^m)&=&\frac1{3}\frac{3\sigma_1^2-\sigma_2}{4},\nonumber\\
f_3({\bf d}^m)&=&\frac1{4}\frac{\sigma_1^2-\sigma_2}{2}\;\sigma_1,\nonumber\\
f_4({\bf d}^m)&=&\frac1{5}\frac{15\sigma_1^4-30\sigma_1^2\sigma_2+5\sigma_2^2+
2\sigma_4}{48},
\nonumber\\
f_5({\bf d}^m)&=&\frac1{6}\frac{3\sigma_1^4-10\sigma_1^2\sigma_2+5\sigma_2^2+
2\sigma_4}{16}\;\sigma_1,\nonumber\\
f_6({\bf d}^m)&=&\frac1{7}\frac{63\sigma_1^6-315\sigma_1^4\sigma_2+315\sigma_1^2
\sigma_2^2+126\sigma_1^2\sigma_4-35\sigma_2^3-42\sigma_2\sigma_4-16\sigma_6}
{576},\nonumber\\
f_7({\bf d}^m)&=&\frac1{8}\frac{9\sigma_1^6-63\sigma_1^4\sigma_2+105\sigma_1^2
\sigma_2^2+42\sigma_1^2\sigma_4-35\sigma_2^3-42\sigma_2\sigma_4-16\sigma_6}{144}
\;\sigma_1.\nonumber
\eea
An absence of power sums $\sigma_k$ with odd indices $k$ are strongly related 
to the presence of Bernoulli's numbers ${\cal B}_r$ in formula (\ref{j17}). A 
simple comparison of formulas (\ref{j15}) and (\ref{j18}) manifests a visual 
similarity between polynomials $T_r({\bf x}^m)$ and $f_r({\bf d}^m)$, which we 
resume in the next conjecture.
%%%%%%%%%%%%%%%%%%%%%%%%%%%%%%%%%%%%%%%%%%%%%%%%%
\begin{conjecture}\label{con1}
Let $T_r({\bf x}^m)$ and  $f_r({\bf x}^m)$ be symmetric polynomials, defined in 
(\ref{j13}) and (\ref{j17}), respectively. Then, the following relation holds
\bea
T_r(E_1,E_2,\ldots,E_r)=f_r(E_1,-E_2,\ldots,-E_r),\quad r\ge 2,\label{j19}
\eea
where signs of arguments $E_j$ are changed only at $E_2,\ldots,E_r$.
\end{conjecture}
%%%%%%%%%%%%%%%%%%%%%%%%%%%%%%%%%%%%%%%%%%%%%%%%%%%%%%%%
\section{Parity properties of $W_1\!\left(s,{\bf d}^m\right)$ and 
generalization of identities for $T_r({\bf x}^m)$}\label{l3}
%%%%%%%%%%%%%%%%%%%%%%%%%%%%%%%%%%%%%%%%%%%%%%%%%%%%%%%%
The polynomials $T_r({\bf x}^m)$ and $f_r({\bf d}^m)$ possess one more kind of 
similarity besides of formulas in (\ref{j15},\ref{j18}). It is easy to verify 
that identities (\ref{j16}) hold for functions $f_r({\bf d}^m)$ by replacing 
$T_r({\bf x}^m)\to f_r({\bf d}^m)$. Keeping in mind such similarity, let us 
find a general form of identities for $f_r({\bf d}^m)$. Making use of a 
recursive relation in \cite{rf06}, formula (12), for their generating function 
$W_1\!\left(s,{\bf d}^m\right)$,
\bea
W_1\!\left(s,{\bf d}^m\right)=W_1\!\left(s-d_m,{\bf d}^m\right)+W_1\!\left(s,
{\bf d}^{m-1}\right),\qquad{\bf d}^{m-1}=\{d_1,\ldots,d_{m-1}\},\label{j20}
\eea
prove the parity properties 
\bea
W_1\!\left(s\!-\!\frac{\sigma_1}{2},{\bf d}^{2m}\right)\!=\!-W_1\!\left(-s\!-\!
\frac{\sigma_1}{2},{\bf d}^{2m}\right),\quad W_1\!\left(s\!-\!\frac{\sigma_1}
{2},{\bf d}^{2m+1}\right)\!=\!W_1\!\left(-s\!-\!\frac{\sigma_1}{2},{\bf d}^
{2m+1}\right),\;\;\label{j21}
\eea
following a similar proof for the whole partition function $W\left(s,{\bf d}^m
\right)$ in \cite{fr02}, Lemma 4.1. Indeed, the recursive relation (\ref{j20}) 
may be rewritten for $V_1\left(s,{\bf d}^m\right)=W_1\left(s-\sigma_1/2,{\bf d}
^m\right)$, where $\sigma_1/2=f_1({\bf d}^m)$,
\bea 
V_1\left(s,{\bf d}^m\right)=V_1\left(s-d_m,{\bf d}^m\right)+V_1\left(s-
\frac{d_m}{2},{\bf d}^{m-1}\right).\nonumber
\eea
Making use of a new variable $q=s-d_m/2$, the last relation reads
\bea
V_1\left(q,{\bf d}^{m-1}\right)&\!=\!&V_1\left(q+\frac{d_m}{2},{\bf d}^m\right)
-V_1\left(q-\frac{d_m}{2},{\bf d}^m\right),\nonumber\\
-V_1\left(-q,{\bf d}^{m-1}\right)&\!=\!&V_1\left(-q-\frac{d_m}{2},{\bf d}^m
\right)-V_1\left(-q+\frac{d_m}{2},{\bf d}^m\right).\nonumber
\eea
Hence, if $V_1\!\left(q,{\bf d}^m\right)$ is an even function of $q$, then $V_1
\!\left(q,{\bf d}^{m-1}\right)$ is an odd one, and vice versa. But, according 
to (\ref{j17}), for $m=1$ we have $V_1\!\left(q,{\bf d}^1\right)=W_1\!\left(q-
d_1/2,{\bf d}^1\right)=1/d_1$, where ${\bf d}^1=\{d_1\}$, or in other words, 
the function $V_1\!\left(q,{\bf d}^1\right)$ is even in $q$. Therefore we obtain
\bea
V_1\left(s,{\bf d}^{2m}\right)=-V_1\left(-s,{\bf d}^{2m}\right),\qquad
V_1\left(s,{\bf d}^{2m+1}\right)=V_1\left(-s,{\bf d}^{2m+1}\right),\nonumber
\eea
that finally leads to (\ref{j21}). 

Identities (\ref{j21}) impose a set of relations on $f_r({\bf d}^m)$. To find 
them, we have to cancel in a series expansion (\ref{j17}) for $W_1\!\left(s\!-
\!f_1({\bf d}^{2m}),{\bf d}^{2m}\right)$ all terms with even degrees of $s$
\bea
s^{2m-1-r}\sum_{k=0}^r(-1)^k{2m-1\choose r-k}{2m-1-r+k\choose k}f_1^k\left(
{\bf d}^{2m}\right)f_{r-k}\left({\bf d}^{2m}\right),\label{j22}
\eea
and for $W_1\!\left(s\!-\!f_1({\bf d}^{2m+1}),{\bf d}^{2m+1}\right)$ all terms 
with odd degrees of $s$
\bea
s^{2m-r}\sum_{k=0}^r(-1)^k{2m\choose r-k}{2m-r+k\choose k}f_1^k\left({\bf d}^
{2m+1}\right)f_{r-k}\left({\bf d}^{2m+1}\right).\label{j23}
\eea
Making use of identity for binomial coefficients
\bea
{A-1\choose B-1-C}{A-B+C\choose C}={A-1\choose B-1}{B-1\choose C},
\qquad A>B>C\ge 0,\nonumber
\eea
and substituting $r=2n-1$ into (\ref{j22}) and (\ref{j23}), and equating them 
to zero, we obtain, respectively,
\bea
s^{2(m-n)}{2m-1\choose 2n-1}\sum_{k=0}^{2n-1}(-1)^k{2n-1\choose k}f_1^k\left(
{\bf d}^{2m}\right)f_{2n-1-k}\left({\bf d}^{2m}\right)=0,\;\;\label{j24}\\
s^{2(m-n)+1}{2m\choose 2n-1}\sum_{k=0}^{2n-1}(-1)^k{2n-1\choose k}f_1^k\left(
{\bf d}^{2m+1}\right)f_{2n-1-k}\left({\bf d}^{2m+1}\right)=0.\;\;\label{j25}
\eea
By comparison (\ref{j24}) and (\ref{j25}) and keeping in mind $f_1\left({\bf 
d}^m\right)\ne 0$, we arrive at universal relation irrespectively to the parity
of $m$,
\bea
\frac{f_{2n-1}\left({\bf d}^m\right)}{f_1^{2n-1}\left({\bf d}^m\right)}=\sum
_{k=1}^{2n-1}(-1)^{k+1}{2n-1\choose k}\frac{f_{2n-1-k}\left({\bf d}^m\right)}
{f_1^{2n-1-k}\left({\bf d}^m\right)},\qquad 1\le n\le\frac{m}{2}.\label{j26}
\eea
Note, that for $n=1$ equality (\ref{j26}) holds identically. Applying a 
recursive procedure to formula (\ref{j26}), the last expression may be 
represented as follows,
\bea
\frac{f_{2n-1}\left({\bf d}^m\right)}{f_1^{2n-1}\left({\bf d}^m\right)}&
\!\!\!\!=\!\!\!\!&\!\!
\sum_{k_1=1}^n{2n-1\choose 2k_1-1}\frac{f_{2(n-k_1)}\left({\bf d}^m\right)}
{f_1^{2(n-k_1)}\left({\bf d}^m\right)}-\label{j27}\\
&&\!\!\sum_{k_1,k_2=1}^n\!\!{2n-1\choose 2k_1}{2(n-k_1)-1\choose 2k_2-1}
\frac{f_{2(n-k_1-k_2)}\left({\bf d}^m\right)}{f_1^{2(n-k_1-k_2)}
\left({\bf d}^m\right)}+\nonumber\\
&&\!\!\!\!\sum_{k_1,k_2,k_3=1}^n\!\!\!{2n-1\choose 2k_1}\!{2(n-k_1)-1\choose 
2k_2}\!{2(n-k_1-k_2)-1\choose 2k_3-1}\!\frac{f_{2(n\!-\!k_1\!-\!k_2\!-\!k_3)}
\left({\bf d}^m\right)}{f_1^{2(n\!-\!k_1\!-\!k_2\!-\!k_3)}\left({\bf d}^m\right)
}\!-\ldots\nonumber
\eea
where a number of summation is equal $n$. Finally, formula (\ref{j27}) may be 
presented in a more simple way
\bea
\frac{f_{2n-1}\left({\bf d}^m\right)}{f_1^{2n-1}\left({\bf d}^m\right)}=
\sum_{r=1}^n(-1)^{r+1}C_{n,r}\;\frac{f_{2(n-r)}({\bf d}^m)}{f_1^{2(n-r)}
({\bf d}^m)},\qquad C_{n,r}\in {\mathbb Z}_>,\label{j28}
\eea
where coefficients $C_{n,r}$ with $r=1,2,3,4$ are calculated below
\bea
C_{n,1}&\!\!\!\!=\!\!\!\!&{2n-1\choose 1},\label{j29}\\
C_{n,2}&\!\!\!\!=\!\!\!\!&{2n-1\choose 2}{2n-3\choose 1}-{2n-1\choose 3},
\nonumber\\
C_{n,3}&\!\!\!\!=\!\!\!\!&{2n-1\choose 2}{2n-3\choose 2}{2n-5\choose 1}\!-\!
{2n-1\choose 2}{2n-3\choose 3}\!-\!{2n-1\choose 4}{2n-5\choose 1}\!+\!
{2n-1\choose 5},\nonumber\\
C_{n,4}&\!\!\!\!=\!\!\!\!&{2n-1\choose 2}{2n-3\choose 2}{2n-5\choose 2}
{2n-7\choose 1}\!-\!{2n-1\choose 2}{2n-3\choose 4}{2n-7\choose 1}\!-
\nonumber\\
&&{2n-1\choose 4}{2n-5\choose 2}{2n-7\choose 1}\!-\!{2n-1\choose 2}
{2n-3\choose 2}{2n-5\choose 3}\!+\nonumber\\
&&{2n-1\choose 2}{2n-3\choose 5}\!+\!{2n-1\choose 4}{2n-5\choose 3}\!+\!
{2n-1\choose 6}{2n-7\choose 1}\!-\!{2n-1\choose 7},\nonumber
\eea
and the higher $C_{n,r}$ have to be determined recursively by (\ref{j27}). The 
total number of terms (products of binomial coefficients) contributing to 
formula (\ref{j29}) for $C_{n,r}$ is given by $2^{r-1}$.

It is easy to verify that formulas (\ref{j28}) do nicely provide the integer 
coefficients in (\ref{j16}) for $n=2,3,4$ successively. That observation make 
us to pose the next conjecture.
%%%%%%%%%%%%%%%%%%%%%%%%%%%%%%%%%%%%%%%%%%%%%%%%%
\begin{conjecture}\label{con2}
Let $T_r({\bf x}^m)$ be symmetric polynomials, defined in (\ref{j13}), then 
$T_r({\bf x}^m)$ satisfy the following identities,
\bea
\frac{T_{2n-1}\left({\bf x}^m\right)}{T_1^{2n-1}\left({\bf x}^m\right)}=\sum_
{k=1}^{2n-1}(-1)^{k+1}{2n-1\choose k}\frac{T_{2n-1-k}\left({\bf x}^m\right)}
{T_1^{2n-1-k}\left({\bf x}^m\right)},\qquad 1\le n\le\frac{m}{2}.\label{i30}
\eea
\end{conjecture}
%%%%%%%%%%%%%%%%%%%%%%%%%%%%%%%%%%%%%%%%%%%%%%%%%
\section{Miscellaneous}\label{l4}
%%%%%%%%%%%%%%%%%%%%%%%%%%%%%%%%%%%%%%%%%%%%%%%%%
In this section we give a list of double inequalities for a ratio of 
polynomials, $T_r({\bf x}^m)/T_1^r({\bf x}^m)$, $r\le 7$, by applying a simple 
version of Maclaurin's inequalities \cite{har59} for power sums $E_r$,
\bea
\frac1{m}E_1^2\le E_2\le E_1^2,\hspace{.5cm}\frac1{m}E_2^2\le E_4\le E_2^2,
\hspace{.5cm}\frac1{m}E_3^2\le E_6\le E_3^2,\hspace{.5cm}\frac{E_1^3}{m^3}\le
\frac1{m}\frac{E_2^2}{E_1}\le\frac{E_4}{E_1}\le E_3\le E_1^3,\nonumber
\eea
to the expressions in (\ref{j15}),
\bea
1+\frac1{3m}&\le\frac{T_2({\bf x}^m)}{T_1^2({\bf x}^m)}\le&\frac{4}{3},
\label{j31}\\
1+\frac1{m}&\le\frac{T_3({\bf x}^m)}{T_1^3({\bf x}^m)}\le&2,\nonumber\\
\frac{13}{15}+\frac{2}{m}+\frac1{3m^2}&\le\frac{T_4({\bf x}^m)}{T_1^4({\bf x}
^m)}\le&\frac{2}{3}\left(5-\frac1{5m}\right),\nonumber\\
\frac1{3}+\frac{10}{3m}+\frac{5}{3m^2}&\le\frac{T_5({\bf x}^m)}{T_1^5({\bf x}
^m)}\le&2\left(3-\frac1{3m}\right),\nonumber\\
\frac{8}{9}+\frac{5}{m}+\frac{3}{m^2}+\frac{16}{63m^7}&\le\frac{T_6({\bf x}^m)}
{T_1^6({\bf x}^m)}\le&\frac{8}{3}\left(\frac{31}{7}-\frac1{m}\right),\nonumber\\
\frac{2}{9}+\frac{7}{m}+\frac{7}{m^2}+\frac{16}{9m^7}&\le\frac{T_7({\bf x}^m)}
{T_1^7({\bf x}^m)}\le&\frac{4}{3}\left(19-\frac{7}{m}\right).\nonumber
\eea

We finish the paper with general expressions for polynomials $T_r({\bf x}^m)$ 
of small $m=1,2,3$ and supplementary relations for power sums $E_k$.
\bea
m=1,\qquad T_{r-1}({\bf x}^1)=\frac{x^{r-1}}{r},\qquad\frac{T_r({\bf x}^1)}
{T_1^r({\bf x}^1)}=\frac{2^r}{r+1},\qquad E_k=E_1^k.\label{j32}
\eea
It is easy verify by (\ref{j31}) that the lower and upper bounds coincide when 
$m=1$ and satisfies (\ref{j32}).
\bea
m=2,\qquad\frac{T_{r-2}({\bf x}^2)}{(r-2)!}=\sum_{k_1,k_2=1\atop k_1+k_2=r}
^{r-1}\frac{x_1^{k_1-1}x_2^{k_2-1}}{k_1!\;k_2!},\hspace{3.5cm}\label{j33}
\eea
\vspace{-.3cm}
\bea
&&E_3=\frac1{2}\left(3E_2-E_1^2\right)E_1,\qquad E_4=E_1^2E_2+\frac1{2}\left(
E_2^2-E_1^4\right),\qquad E_5=\frac1{4}\left(5E_2^2-E_1^4\right)E_1,\qquad
\nonumber\\
&&E_6=\frac1{4}\left(E_2^2+6E_1^2E_2-3E_1^4\right)E_2.\nonumber
\eea
\bea m=3,\qquad\frac{T_{r-3}({\bf x}^3)}{(r-3)!}=\sum_{k_1,k_2,k_3=1\atop 
k_1+k_2+k_3=r}^{r-2}\frac{x_1^{k_1-1}x_2^{k_2-1}x_3^{k_3-1}}{k_1!\;k_2!\;k_3!},
\hspace{2cm}\label{j34}
\eea
\vspace{-.3cm}
\bea
&&E_4=\frac1{6}\left(E_1^4+3E_2^2-6E_1^2E_2+8E_1E_3\right),\qquad E_5=\frac1{6}
\left(E_1^5-5E_1^3E_2+5E_1^2E_3+5E_2E_3\right),\nonumber\\
&&E_6=\frac1{12}\left(E_1^6+2E_2^3+4E_3^2-9E_1^2E_2^2+12E_1E_2E_3-3E_1^4E_2+
4E_1^3E_3\right).\nonumber
\eea

In regards with Conjecture \ref{con1}, which is left open, it would be worth to 
find such a transformation of expression (\ref{j33},\ref{j34}) and its 
generalization on arbitrary $m$
\bea
\frac{T_{r-m}({\bf x}^m)}{(r-m)!}=\sum_{k_1,\ldots,k_m=1\atop k_1+\ldots +k_m
=r}^{r-m+1}\prod_{j=1}^m\frac{x_j^{k_j-1}}{k_j!},\nonumber
\eea
which is similar to (\ref{j17}) including Bernoulli's numbers.
%%%%%%%%%%%%%%%%%%%%%%%%%%%%%%%%%%%%%%%%%%%%%%%%%%%%%%%%%%%

\end{document}